\documentclass[11pt]{article}
\usepackage{amsmath,amssymb,amsthm,amsfonts,mathrsfs,mathtools,physics}

\usepackage[nottoc]{tocbibind}
\usepackage{dsfont} %Fonctions indicatrices avec mathds
\usepackage[utf8]{inputenc}
\usepackage[T1]{fontenc}
\usepackage[english]{babel}
\usepackage{palatino}
\usepackage{comment}
\usepackage{fullpage}
\usepackage{enumerate}
\usepackage[toc]{appendix}
\usepackage{esint}
\usepackage{appendix}
\usepackage{url}
\usepackage{graphics,graphicx}
\usepackage{ifthen}
\usepackage[normalem]{ulem} %to strike the words
\usepackage{caption}
\usepackage{tikz}

\usepackage[backref=page]{hyperref}
 \hypersetup{
 	plainpages=false,
 	unicode=false,          % non-Latin characters in Acrobat's bookmarks
 	pdftoolbar=true,        % show Acrobat's toolbar?
 	pdfmenubar=true,        % show Acrobat's menu?
 	pdffitwindow=true,      % page fit to window when opened
 	pdftitle={My title},    % title
 	pdfauthor={},     % author
 	pdfsubject={},   % subject of the document
 	pdfcreator={},   % creator of &  the document
 	pdfproducer={}, % producer of the document
 	pdfkeywords={}, % list of keywords
 	pdfnewwindow=true,      % links in new window
 	colorlinks=true,       % false: boxed links; true: colored links   %%%%%%%%%%% MODIFIE
 	linkcolor=blue,          % color of internal links
 	citecolor=blue,        % color of links to bibliography
 	filecolor=blue,      % color of file links
 	urlcolor=blue,          % color of external links
 	urlbordercolor=0 1 1
 }
\usepackage{xcolor}
\definecolor{Blue}{rgb}{0.,0.,1}

\usepackage{setspace}
\usepackage[pass]{geometry}

\newcommand{\R}{\mathbb R}
\newcommand{\N}{\mathbb N}

\newcommand{\cH}{\mathcal{H}}

\newcommand{\cL}{\mathcal{L}}

\newcommand{\G}{{\mathbb G}_{d,n}}
\newcommand{\V}{\|V\|}

\newcommand{\M}{\|M\|}
\newtheorem{theo}{Theorem}[section]
\newtheorem*{theo*}{Theorem}
\newtheorem{prop}[theo]{Proposition}
\newtheorem{lemma}[theo]{Lemma}
\newtheorem{dfn}[theo]{Definition}
\newtheorem*{dfn*}{Definition}

\newtheoremstyle{rmdotless}{}{}{\upshape}{}{\bfseries}{.}{0.5em}{}
\theoremstyle{rmdotless}
\newtheorem{remk}[theo]{Remark}

\DeclareMathOperator{\mdiv}{div}
\DeclareMathOperator*{\argmin}{arg\,min}
\DeclareMathOperator*{\supp}{spt}

\DeclareMathOperator*{\lip}{Lip}

\renewcommand{\phi}{\varphi}
\renewcommand{\epsilon}{\varepsilon}
\newcommand{\e}{\epsilon}
\renewcommand{\G}{G_{d,n}}

\numberwithin{equation}{section} %equation take the number if the section
\title{ Integral Brakke approximate equality \\for an approximate mean curvature flow}

\author{}
\date{}
\begin{document}
\maketitle
\begin{center}
 {\bf Abdelmouksit Sagueni}\\
Institut Camille Jordan\\
{\it sagueni@math.univ-lyon1.fr}
\end{center}
\begin{abstract}
% The mean curvature flow is a geometric flow allowing to reduce the intrinsic volume of submanifolds. In his work \cite{brakke}, K.Brakke proposed a weak formulation of this evolution in the geometric measure theoretic setting. The idea was to derive an inequality involving the mass and the velocity vector (the mean curvature) which characterizes the evolution in the smooth case to a certain extent. This inequality was then used to provide a weak definition for the evolution by the mean curvature flow for less regular initial data, such as rectifiable varifolds.\\
In this paper we aim to study the consistency of the mean curvature flow via discretization. We will use discretizations by volumetric varifolds, and derive a Brakke approximate equality involving the masses of the volumetric varifolds and their approximate mean curvatures.

% (defined as in \cite{blm1}).

\vspace{0.5cm}
{\bf Key words:} Geometric Measure Theory; Varifolds; Volumetric discretization; Mean curvature; Mean curvature flow; Brakke flow.
\end{abstract}

\tableofcontents

\vfill
\noindent\rule{\textwidth}{0.4pt} % horizontal bar
\textbf{Acknowledgments:} I would like to thank my thesis co-advisor, Blanche Buet, for her suggestions and invaluable support.

\newpage

\section{Introduction}
\subsection{Notations and preliminaries}
Throughout the paper, we fix $d,n \in \N$ such that  $1\leq d \leq n$ and we adopt the following notations:
\begin{itemize}
\item $\cL^n$ denotes the $n$-dimensional Lebesgue measure in $\R^n$.
\item $\cH^d$ denotes the $d$-dimensional Hausdorff measure in $\R^n$.
\item $\omega_d$ denotes the volume of the unit ball of $\R^d$. 
\item For a set $A$, $A^{\delta}=\bigcup\limits_{x\in A}B_{\delta}(x)=\lbrace y \in \R^n, d(y,A)<\delta \rbrace$.
%\item $C^0(X,Y)$ will denote the space of continuous functions from the space $X$ to the space $Y$.
 \item For a function $u$, define the supremum norm as: $||u||_{\infty}=\sup\limits_x|u(x)|$.
 \item $C^0, C_c^0$ denote the space of continuous functions, continuous compactly supported functions, respectively.
 \item $C_c^{1}$ denotes the space of  continuously differentiable functions of compact support. We define for $\phi \in C_c^{1}$, $||\phi||_{C^1} := ||\phi||_{\infty}+ ||D\phi||_{\infty}.$
 \item $C_c^{2}$ denotes the space of twice continuously differentiable functions of compact support. We define for $\phi \in C_c^{2}$, $||\phi||_{C^2} := ||\phi||_{\infty}+ ||D\phi||_{\infty}+||D^2\phi||_{\infty}.$
%  \item For $X\subset\R^n$, $X^c := \R^n \setminus X$ denotes the complementary subset in $\R^n$.
\item For $\Omega \subset \R^n$, $\underset{\Omega}{\lip}(f)$ denotes the Lipschitz coefficient of $f$ on $\Omega$.
\item We fix $\rho$, a $C^2$ nonnegative real valued function supported on $[0,1]$ and $\xi$ a $C^1$ nonnegative real valued function supported on $[0,1]$ , we choose $\xi$ to be positive on $]0,1[$.
\item For $\nu, \mu$ two Radon measures, $\Delta(\nu,\mu)$ denotes the bounded Lipschitz distance between $\nu$ and $\mu$ defined as:
\begin{equation*}
 \Delta(\nu,\mu) := \sup \Big\lbrace \int \phi \,d\nu - \int \phi \,d\mu, \,\,||\phi||_{\infty}+ \lip(\phi) \leq 1 \Big\rbrace.
\end{equation*}
\item $\G$ is the Grassmannian manifold of $d$-dimensional vector subspaces of $\R^n$.
\item By a closed submanifold, we mean a compact boundaryless submanifold.
\end{itemize}

\subsection{Mean curvature flow, Brakke flow}
Given a closed $C^2$ submanifold $M$ of dimension $d$ in $\R^n$, the mean curvature flow of $M$ is a geometric flow starting from $M$ and satisfying the equation $\partial M(t) = H(M(t))$, where $H(M(t))$ stands for the mean curvature vector of $M$. The mean curvature vector is characterized by the formula: 
\begin{equation}
 \frac{d}{dt}(\cH^d(M+tX))_{|t=0} = -\int_{M} H(y,M) \cdot X(y) \, d\cH^d(y)
\end{equation}
for any $X \in C^1_c(\R^n,\R^n)$. The formula translates the fact that the mean curvature is the direction corresponding to the steepest descent of the area functional. Alternatively, if we define for any $C^1$ vector $X$ and  any $d$-plane $S$, 
\begin{equation}
\mdiv_S(X) :=\sum\limits_{i=1}^d \langle \nabla(X\cdot \tau_i), \tau_i \rangle                                                                                                                                                                                                                                                                                                                                                                                                                                                                                                  \end{equation}
where $\lbrace \tau_i \rbrace_{i=1}^d$ is an orthonormal basis of $S$, the mean curvature vector takes the following form:
\begin{equation}
  H(x,M) := \sum\limits_{j=1}^{n-d} \mdiv_{T_xM}(\nu_j)\nu_j,
\end{equation}
for any $\lbrace \nu_j \rbrace_{j=1}^{n-d}$, orthonormal basis of $T^{\perp}_xM$.\\
By the theory of partial differential equations, the mean curvature flow equation admits a unique solution, defined for an interval of time $[0,T), T \in (0,\infty)$ depending on the geometry of the initial submanifold. The flow has been widely studied in the differential geometric setting, we refer the reader to the works of Grayson \cite{gra} and Huisken et al. \cite{hui}, \cite{eckhui} . In general, solving the mean curvature flow equation for $T=+\infty$ is not possible and the flow develops singularities at a finite time. To extend the definition of the mean curvature flow beyond singularities, several methods were proposed, we mention: \cite{brakke}, \cite{cgg}, \cite{evsp1}, \cite{Ilmanen} and \cite{bel}), the Brakke's method is the one that we will be dealing with in this paper.\\
The idea behind Brakke's solution to the mean curvature flow is to seek an integral version of the mean curvature equation (see \cite[section 2.1]{ton} for more details). In the measure theoretic setting, the quantities 
\begin{equation}
 \Big\lbrace \int_{M(t)} \phi(y) \, d\cH^d(y),  \,\, \phi \in C_c(\R^n,\R^+) \Big\rbrace_t.
\end{equation}
fully characterizes the family $\lbrace M(t) \rbrace_t$. Differentiating  when $\phi\in C_c^1(\R^n,\R^+)$ gives:
\begin{equation}
 \frac{d}{dt}  \int_{M(t)} \phi(y) \,d\cH^d(y) = \int_{M(t)} - \phi(y) |H(y,M(t))|^2 +\nabla\phi(y) \cdot H(y,M(t)) \, d\cH^d(y).
\end{equation}
Integrating on $[t_1,t_2] \subset [0,T)$, we get:
\begin{equation}\label{weakmcfequation}
 \int_{M(t_1)}\phi(y) \, d\cH^d(y)  - \int_{M(t_2)}\phi(y) \, d\cH^d(y) = \int_{t_1}^{t_2}\int_{M(t)} - \phi(y) |H(y,M(t))|^2 + \nabla\phi(y) \cdot H(y,M(t)) \, d\cH^d(y) \,dt.
\end{equation}
for any $\phi\in C_c^1(\R^n, \R^+)$ and $0\leq t_1 \leq t_2 < T$. One could prove that (see \cite[chapter 2]{ton}), any family of Radon measures $\lbrace M(t) \rbrace_t$ whose space-time track is $C^2$, satisfying \eqref{weakmcfequation} is a classical mean curvature flow, this proves the consistency in the regular case. Brakke, in his ingenious work \cite{brakke} showed that, starting from any rectifiable varifold $V_0$, there exists a family of rectifiable varifolds $V_t, t\in [0,+\infty]$ satisfying the "Brakke inequality" \eqref{weakmcfequation}. We note that Brakke worked with an inequality instead $(\leq)$ and the reason behind is beyond the scope of this paper.
\subsection{Varifolds, discrete varifolds and approximation}
The central tool used in our work is the notion of varifold.
\begin{dfn}(Varifold)\label{varifolds}
A $d$-varifold in $\R^n$ is a nonnegative Radon measure on $\R^n \times \G$. The space of $d$-varifolds in $\R^n$ is denoted $V_d(\R^n)$.
\end{dfn}
We also consider the projection measure on $\R^n$, we call it the mass measure of the varifold $V$ and we denote it by $||V||$. It is defined as follows:
\begin{equation}
||V||(\phi):=\int_{\R^n\times \G}\phi(x) \,dV(x,S), \quad \forall \phi \in C_c(\R^n,\R^+).
\end{equation}
In other words, for every Borel set $A\subset \R^n$, we have: $||V||(A)=V(A\times \G)$.\\
In the following, we list some of the most used/familiar types of varifolds.
\begin{enumerate}
 \item Smooth varifolds: to a $d$-dimensional submanifold $M$ in $\R^n$, we associate the varifold:
 \begin{equation}
   \phi \mapsto  \int_{M} \phi(y,T_yM) \,d\cH^d(y) \quad  \forall \phi\in C^0_c(\R^n \times \G,\R^+).
 \end{equation}
 Then, the associated mass varifold is defined as: 
 \begin{equation}
 \phi \mapsto  \int_{M} \phi(y) \,d\cH^d(y) \quad  \forall \phi\in C^0_c(\R^n,\R^+).
 \end{equation}
 \item Rectifiable varifolds: to a $d$-dimensional rectifiable set $M$ (see \cite[chap 3]{simon} for definition) in $\R^n$, and a $\cH^d$- integrable nonnegative function $\theta$ on $\supp M$ , we associate the varifold 
 \begin{equation}
   \phi \mapsto \int_{M} \phi(y,T_yM) \theta(y)\, d\cH^d(y) \quad  \forall \phi\in C^0_c(\R^n \times \G,\R^+).
 \end{equation}
  Then, the associated mass varifold is defined as: 
\begin{equation}
   \phi \mapsto \int_{M} \phi(y) \theta(y) \,d\cH^d(y) \quad  \forall \phi\in C^0_c(\R^n,\R^+).
 \end{equation}
  
 \item Point cloud varifolds: to a distribution of points $\lbrace x_j \rbrace_{j=1}^{N}$ in $\R^n$, $d$-tangents $\lbrace P_j \rbrace_{j=1}^{N}$ in $\G$ and masses $\lbrace m_j \rbrace_{j=1}^{N}$ in $\R^+$, we associate the varifold:
 \begin{equation}
  \phi \mapsto  \sum\limits_{j=1}^{N} m_j \phi(x_j,P_j)  \quad  \forall \phi\in C^0_c(\R^n \times \G,\R^+).
 \end{equation}
Then, the associated mass varifold is defined as: 
\begin{equation}
  \phi \mapsto  \sum\limits_{j=1}^{N} m_j \phi(x_j)  \quad  \forall \phi\in C^0_c(\R^n,\R^+).
\end{equation}
\end{enumerate}
We recall the notion of Ahlfors regularity in the setting of varifolds.
\begin{dfn}\label{dfnahlfors}(Ahlfors regularity)
Let $V \in V_d(\R^n)$, we say that $V$ is Ahlfors regular if its mass measure $||V||$ is, i.e. there exits $C_0 > 1$, $r_0 > 0$ such that for all $x\in \supp \V$ and $0 < r \leq r_0$,
\begin{equation}\label{ahlfors}
C_0^{-1} r^d \leq \V(B(x,r)) \leq r^d C_0.
\end{equation}
\end{dfn}
We note that $r$ can be chosen as large as needed: if condition \eqref{ahlfors} holds for some $r_0 > 0$ then it holds for any $r \geq r_0$ possibly adapting the regularity constant $C_0$.

\subsubsection{Approximation by discrete varifolds}
Volumetric varifolds were introduced to discretize submanifolds in the spirit of varifolds, it was proved to be suitable for discretizing submanifolds of any dimension and co-dimension and to handle the presence of singularities. Here, we present some of the results of \cite{buet} on the discretization and the approximation by volumetric varifolds.\\ 
Given $\Omega$, a bounded open set of $\R^n$ and a mesh $\mathcal{K}$ of size $h$, we understand by that: a collection of cells covering $\Omega$ where $h$ bounds the diameter of any of the cells. Given a $d$-submanifold  $M$ in $\Omega$, we define a volumetric discretization $V_h$ of $M$ as follows:
\begin{equation}\label{dfnvolumetric}
V_h = \sum\limits_{K\in \mathcal{K}} \frac{m_K}{|K|}\cL^n_{|K}\otimes\delta_{P_K}   \quad 
\end{equation}
\text{where:}$\,\,|K| =  \cL^n(K),\,\,  m_K = \cH^d(M\cap K),\,$
  and $\displaystyle P_K \in \argmin\limits_{S\in\G}\int_{M \cap K}|T_yM-S|\,d\cH^d(y)$
for any cell $K$ in $\mathcal{K}$.
% We define a points clouds varifold $V'_h$ associated to the mesh $K_h$ as follows:
% \begin{equation}
%  V'_h = \sum\limits_{K\in \mathcal{K}} m_K \delta_{x_K} \otimes\delta_{P_K}  
% \end{equation}
% where $m_K$ and $P_K$ defined same as previous and $x_K$ is chosen arbitray in the cell $K$.\\
% We have the next Proposition on the proximety of volumetric and point clouds discretization.
% \begin{prop}
%  Given $\Omega$ a bounded open set of $\R^n$, a mesh $\mathcal{K}$ of size $h$. For any collection of masses $m_K$'s, tangents $P_K$'s and points $\lbrace x_K \rbrace_K$ where $ x_K \in K$ for any $K$, define 
%  \begin{equation}
%   V_h = \sum\limits_{K\in \mathcal{K}} \frac{m_K}{|K|}\cL^n_{|K}\otimes\delta_{P_K},
%  \end{equation}
% \begin{equation}
%   V'_h = \sum\limits_{K\in \mathcal{K}} m_K \delta_{x_K} \otimes\delta_{P_K}.
% \end{equation}
% Then we have: for any Lipschitz function $\phi$ in $\R^n\times\G$
% \begin{equation}
%  \Big|V_h(\phi) -V'_h(\phi) \Big|\leq h \lip(\phi) ||V_h||(\R^n).
% \end{equation}
% \end{prop}
% Thus, one could study the properties of dicretization and approximation for either volumetric or point clouds varifolds and transferet them to the other type.
\begin{figure}
\begin{center}
\includegraphics[scale=0.7]{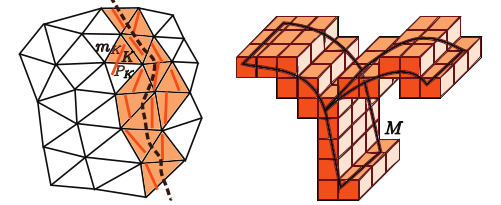}
\caption{\cite{buet} Approximation by volumetric varifolds.}
\end{center}
\end{figure}
The following proposition results from \cite[Theorem 2.1]{buet} when the approximated rectifiable varifold is smooth.
\begin{prop}\label{volumapp}
 Let $\Omega$ be a bounded set of $\R^n$, $\mathcal{K}_h$ a mesh of $\Omega$. Let  $M\subset \Omega$ be a $d$-submanifold and $V_h$ be a volumetric discretization of $M$ defined as in \eqref{dfnvolumetric}. We have, for any  Lipschitz function $\phi$ on $\Omega$:
 \begin{equation}
  \Big| ||M||(\phi) - ||V_h||(\phi) \Big| \leq h \lip\limits_{\Omega}(\phi) ||M||(\supp(\phi)).
 \end{equation}
In addition, for $C > 0$ such that,
\begin{equation}
 \big| T_xM -T_yM \big| \leq C|x-y| \quad \forall x,y \in \supp M,
\end{equation}
and if we denote $\Pi:\Omega \times \G \rightarrow \R^n, (y,S) \mapsto y$, for every Lipschitz function $\phi$ on $\Omega\times\G$, we have 
\begin{equation}
 \Big| M(\phi) - V_h(\phi) \Big| \leq h \lip\limits_{\Omega}(\phi)(1+2C) ||M||(\Pi(\supp(\phi))).
\end{equation}
\end{prop}
We note that in the original statement of \cite[Theorem 2.1]{buet}, $\beta $: the Hölder coefficient of the map $y \mapsto T_yM$ is taken in $(0,1)$, however, same computations involved in the proof shows the result for any $\beta \in \R^+$ (in particular, for $\beta =1$).

\subsubsection{First variation, approximate mean curvature}
To study the  behavior of the measure $V$ w.r.t infinitesimal displacements, we introduce a notion called the \textit{first variation} of the area.
\begin{dfn}[First variation]
Let $X \in C_c^1(\R^n, \R^n)$, $V \in V_d(\R^n)$, we denote 
\begin{equation}\label{firstvar}
\delta V (X) : = \int_{\R^n \times \G} \mdiv_S(X)(x)\, dV(x,S).
\end{equation}
\end{dfn}
The map 
\begin{equation}
 \delta V(\cdot) : C_c^1(\R^n,\R^n) \mapsto \R
\end{equation}
is called the \textit{first variation} of the varifold $V$.\\
If $M$ is a closed $C^2$-submanifold we have: 
\begin{equation}
 \frac{d}{dt}(\cH^d(M+tX))_{|t=0} = \int_{M} \mdiv_{T_yM}(X)(y) \, d\cH^d(y)
\end{equation}
this justifies the naming and the definition of $\delta V$. Moreover, one has:
\begin{equation}
 \int_{M} \mdiv_{T_yM}(X)(y) \, d\cH^d(y) = -\int_{M} H(y,M) \cdot X(y) \, d\cH^d(y).
\end{equation}
Thus, given a varifold $V$, if $\delta V$, seen as a vectorial measure on the space $C^1_0(\R^n,\R^n)$ is $||\cdot||_{\infty}$-bounded, then  by Riesz representation Theorem and the Radon-Nikodym decomposition, there exists a vector $\in L^1 (d||V||)^n$ that we denote $H_V$ such that: for any $X\in C_0(\R^n,\R^n)$, we have:
\begin{equation}
 \delta V(X)=-\int_{\R^n} X(y)\cdot H_V d||V||(y) + \delta V_{s}(X)
\end{equation}
where $\delta V_{s}$ is singular w.r.t $d||V||$, we write $\displaystyle H_V = \frac{\delta V}{||V||}$. $H_V$ is called the generalized mean curvature of the varifold $V$. In general, the measure $\delta V$ is not necessarily bounded, for instance when $V$ is a point clouds varifold. One way to define an approximation of the generalized mean curvature is the following (see \cite{blm1}), consider two smooth nonnegative functions $\rho$ and $\xi$ compactly supported  on $[0,1]$, for $\e \in (0,1]$ define for any $r\in[0,1]$
\begin{equation}
 \rho_{\e}(r) = \e^{-n} \rho\left(\frac{r}{\e}\right), \quad \xi_{\e}=\e^{-n}\xi\left(\frac{r}{\e}\right),
\end{equation}
for any $X\in C_c^1(\R^n,\R^n)$
\begin{equation}
 \left( \delta V \ast \rho_{\e} \right) (X) :=  \delta V(X\ast \rho_{\e}) = \int_{\R^n} \int_{\R^n\times\G} S \left(\nabla\rho_{\e}(z-y) \right) dV(z,S) \cdot X(y)dy,
\end{equation}
we represent the regularized first variation by the vectorial function:
\begin{equation}
 \delta V \ast \rho_{\e}(y) :=  \int_{\R^n\times\G} S \left(\nabla\rho_{\e}(z-y) \right) dV(z,S).
\end{equation}
For any $\phi \in C_c^0(\R^n,\R^+)$
\begin{equation}
\left(  ||V||\ast\xi_{\e} \right) (\phi) := ||V||(\phi\ast\xi_{\e})=\int_{\R^n}\int_{\R^n}\xi_{\e}(z-y)d||V||(z) \phi(y) dy,
\end{equation}
we represent the regularized mass measure by the real valued function:
\begin{equation}
 ||V||\ast\xi_{\e}(y) := \int_{\R^n}\xi_{\e}(z-y)d||V||(z).
\end{equation}
Then, following \cite{blm1}, an approximate mean curvature for $V$ can be defined as:
\begin{equation}
 H_{\rho,\xi,\e}^{V}(y) = -\frac{C_{\rho}}{C_{\xi}}\frac{\delta V \ast \rho_{\e}(y)}{||V||\ast\xi_{\e}(y)} \quad \forall y \in \R^n
\end{equation}
where,
\begin{equation}
 C_{\rho}=d\omega_d\int_0^1 \rho(r)r^{d-1}dr, \quad \text{and} \quad C_{\xi}=d\omega_d\int_0^1 \xi(r)r^{d-1}dr.
\end{equation}
The previous definition of the approximate mean curvature vector enjoys stability and convergence properties (see \cite{blm1},\cite{br}), it is the reason why we will work with the definition in this paper. For simplicity, we will write $H_{\e}(\cdot,V)$ to denote $H_{\rho,\xi,\e}^{V}(\cdot)$ and we will assume that $C_{\rho}=C_{\xi}=1$.
\subsection{Goal of the paper}
We will investigate the consistency of the mean curvature flow via space-discretization by volumetric varifolds. We will consider a smooth mean curvature flow $M(t)$ defined for $t\in[0,T]$, a discretization $V_h(t)$ (of $M(t)$ for every $t\in[0,T]$) and derive a Brakke approximate equality satisfied by the discretization. Concretely, we will replace $M(t)$ by $V_h(t)$ in the equality \eqref{weakmcfequation} and exhibit the resulting error term.
We will consider an ``approximate'' mean curvature vector (instead of the mean curvature) of the discretization $V_h(t)$ since it makes more sense in the discrete setting.

\section{Integral Brakke approximate equality for the discretization}
We know that if a $C^2$-submanifold is contained in a convex domain $\Omega$ then, by the avoidance principle, its mean curvature flow is also contained in $\Omega$ (see \cite[Theorem 3.8]{brakke}). We assume the domain $\Omega$ to be convex to avoid changing the domain/meshing at every time. In the following Theorem we will measure how far a volumetric discretization of a mean curvature flow is from satisfying a Brakke equality:
\begin{theo}\label{thmappbrakke} Let $\e\in(0,1)$, let  $\Omega$ be a convex bounded open domain of $\R^n$. Let $M$ be a $C^3$ closed $d$-submanifold in $\Omega$ and $M(t)$ be its mean curvature flow defined on $[0,T]$. For all $t\in |0,T]$, let $V_h(t)$ be a volumetric discretization of $M(t)$ of parameter $h$ defined as in \eqref{dfnvolumetric}.\\ Moreover, if we assume that  $h$ satisfies: $2h\leq \gamma \e$, and $\gamma$ satisfies:
\begin{itemize}
\item $\gamma \leq (8(1+C_0^{2/d}))^{-1} \quad C_0 $ bounds the Ahlfors regularity constant of $M(t)$ for all $t\in[0,T]$.
\item $ \gamma \leq \max\limits_{x\in M(t), \, t\in [0,T]}(\lambda(x))^{-1}$ \quad \text{where $\lambda(x)$ is the maximal principal curvature at $x$}.
\item $ \beta \geq \gamma 2^{3d}C_0^2(\lip(\xi)+1)$, \quad where $\displaystyle \beta = \min \big\lbrace \xi(s) \big| s\in \big[ \frac{C_0^{-2/d}}{4}, \frac{1}{2} \big]\big\rbrace >0$.
\end{itemize}
then, we have the following estimate,
\begin{equation}\label{brakkevolum}
\begin{split}
   & \Big|  ||V_h(t_2)||(\phi)- ||V_h(t_1)||(\phi) + \int_{t_1}^{t_2}\int_{\R^n} \phi(x) |H_{\e}(x,V_h(t))|^2 - \nabla\phi(x)\cdot H_{\e}(x,V_h(t)) d||V_h(t)||(x)dt \Big| 
   \\& \leq 2 \lip(\phi) \max\limits_{t\in \lbrace t_1,t_2 \rbrace}\Delta(M(t),V_h(t)) + ||\phi||_{\infty} C_1 \e(||M(t_1)||(\R^n)-||M(t_2)||(\R^n)) + ||\phi||_{C^2} C (t_2-t_1) ( \e +  \frac{h}{\e^3}) 
\end{split}
\end{equation}
for every $\phi \in C_c^2(\R^n,\R^+)$ , $0 \leq t_1 \leq t_2 \leq T$, where $C$ depends on $n$, $d$, $\gamma$, $\beta$, $||\rho||_{C^2}$, $||\xi||_{C^1}$, $C_0$, $C_1$(defined in Lemma \ref{hhepsilon1}) and $C_2$, where $C_2$ bounds the Lipschitz constants of the maps $y\mapsto T_yM(t)$ for all $t\in[0,T]$.
\end{theo}
\begin{remk}
 The mean curvature flow is continuous in time w.r.t the $C^{3}$-distance on the space of $d$-submanifold of $\R^n$(see \cite[chapter 3]{eidel}), thus, one can bound uniformly on time, thanks to the compactness of $[0,T]$, the constants: $C_0,C_2$ and the maximum value of the principal curvature. The constant $C_1$ also evolves continuously w.r.t to the $C^{3}$-distance on the space of $d$-submanifolds of $\R^n$, as a consequence, it also can be bounded uniformly on time.
\end{remk}
\begin{remk}\label{thmappbrakke1}
If we ignore the smallness of the right hand side of \eqref{brakkevolum} with respect to the size of the interval $[t_1,t_2]$ we get a weaker (but simpler) estimate:
 \begin{equation*}
  \begin{split}
   & \Big|  ||V_h(t_2)||(\phi)- ||V_h(t_1)||(\phi) + \int_{t_1}^{t_2}\int_{\R^n} \phi(x) |H_{\e}(x,V_h(t))|^2 - \nabla\phi(x)\cdot H_{\e}(x,V_h(t)) d||V_h(t)||(x)dt \Big|
   \\& \leq ||\phi||_{C^2} C' ( \e +  \frac{h}{\e^3}) \quad \text{for} \,\,C'=||M(0)||(\R^n)(2+C_1)+CT,
\end{split}
\end{equation*}
where we used: $\Delta(M(t),V_h(t))\leq h||M(t)||(\R^n)$ and $||M(t)||(\R^n) \leq ||M(0)||(\R^n)$. 
\end{remk}
The proof of this theorem is broken down into several steps, we start with the following lemma which results from [Proposition 3.3 \cite{br}] (see also \cite[paragraph 5]{blm1}).
\begin{lemma}\label{hhepsilon1}
Let $\e\in(0,1)$, and $M$ any $C^3$ closed  $d$-submanifold of $\R^n$. There exists a constant $C_1$ depending on the $C^3$ norm of $M$ (seen locally as a graph over its tangent space) such that: for any $x\in\R^n$, 
 \begin{equation}
      \big| H(x,M) - H_{\e}(x,M) \big| \leq C_1 \e,
 \end{equation}
 \end{lemma}
As stated before, the mean curvature flow is continuous with respect to the $C^3$ distance on the space of submanifolds, using the fact that $[0,T]$ is compact allows to choose $C_1$ uniformly in $t$, i.e.
 \begin{equation}\label{hhepsilon2}
     \big| H(x,M(t)) - H_{\e}(x,M(t)) \big| \leq C_1\e, \quad \forall x\in\R^n, \,\, \forall t\in[0,T].
 \end{equation} 
\begin{remk}
We note that in the original statement of Lemma \ref{hhepsilon1}, the kernels $\rho$ and $\xi$ are supposed to be \textit{natural kernel pair}, i.e. they must satisfy the relation:  $-n\xi(r) = r \rho'(r)$ for all $r\in[0,1]$. Looking carefully at the proof, we realize that this assumption on the pair of kernels is relevant.
\end{remk}
We consequently have the following estimate: 
\begin{prop}\label{HepsilonM}
Let $\e\in(0,1)$, $M$ a $C^3$ closed  $d$-submanifold and $M(t)$ its mean curvature flow, let $\phi \in C^1_c(\R^n,\R^+)$. We have:
\begin{equation}
\begin{split}
   & \Big| ||M(t_2)||(\phi)- ||M(t_1)||(\phi) + \int_{t_1}^{t_2}\int_{\R^n} \phi(x) |H_{\e}(x,M(t))|^2 - \nabla\phi(x)\cdot H_{\e}(x,M(t)) d||M(t)||(x)dt \Big| \\& \leq  ||\phi||_{\infty}C_1\e(||M(t_1)||(\R^n)-||M(t_2)||(\R^n)) + ||\phi||_{C^1}c_3 \e (t_2-t_1),
   \end{split}
\end{equation}
for any $0 \leq t_1 \leq t_2 \leq T$ and $\phi \in C^1_c(\R^n,\R^+)$, where $\displaystyle c_3= C_1 ||M(0)||(\R^n) (2+C_1).$
\end{prop}
\begin{proof}
Let $t_1,t_2\in [0,T]$ such that $0\leq t_1 \leq t_2 \leq T$, let $\phi \in C^1_c(\R^n,\R^+)$.  We have by \eqref{hhepsilon2},
\begin{equation}\label{hhepsilon3}
    \begin{split}
       & \big| \int_{t_1}^{t_2}\int_{\R^n} \nabla\phi(x)\cdot H_{\e}(x,M(t)) d||M(t)||(x)dt - \int_{t_1}^{t_2}\int_{\R^n} \nabla\phi(x)\cdot H_{\e}(x,M(t)) d||M(t)||(x)dt \big|
       \\& \leq ||\nabla\phi||_{\infty} \int_{t_1}^{t_2}\int_{\R^n}  \big| H(x,M(t)) - H_{\e}(x,M(t)) \big| d||M(t)||(x)dt 
       \\& \leq ||\nabla\phi||_{\infty} C_1 \e \int_{t_1}^{t_2}\int_{\R^n} d||M(t)||(x)dt \leq ||\nabla\phi||_{\infty} C_1 ||M(0)||(\R^n) (t_2-t_1) \e,
    \end{split}
\end{equation}
where we used $||M(t)||(\R^n)\leq ||M(0)||(\R^n), \forall t \in[0,T]$. Also, by \eqref{hhepsilon2},
\begin{equation}
\begin{split}
& \Big| \int_{t_1}^{t_2}\int_{\R^n} \phi(x) |H(x,M(t))|^2 d||M(t)||(x)dt - \int_{t_1}^{t_2}\int_{\R^n} \phi(x) |H_{\e}(x,M(t))|^2 d||M(t)||(x)dt \Big|
\\&\leq 
\int_{t_1}^{t_2}\int_{\R^n} \big| \phi(x) |H(x,M(t))|^2 -\phi(x) |H_{\e}(x,M(t))|^2 \big| d||M(t)||(x)dt
\\&\leq
\int_{t_1}^{t_2}\int_{\R^n}  \phi(x)\big| H(x,M(t))-H_{\e}(x,M(t)) \big|\big| H(x,M(t))+H_{\e}(x,M(t)) \big| d||M(t)||(x)dt
\\&\leq ||\phi||_{\infty}
C_1 \e \int_{t_1}^{t_2}\int_{\R^n}  \big| H(x,M(t))+H_{\e}(x,M(t)) \big| d||M(t)||(x)dt 
\\& \leq  ||\phi||_{\infty}
C_1 \e \int_{t_1}^{t_2}\int_{\R^n}  \left( \big| H(x,M(t))\big| + C_1 \e \right)d||M(t)||(x)dt 
\\& \leq ||\phi||_{\infty}
 C_1 \e \int_{t_1}^{t_2}\int_{\R^n} \big| H(x,M(t))\big|  d||M(t)||(x)dt + ||\phi||_{\infty} (C_1\e)^2 \int_{t_1}^{t_2}\int_{\R^n} d||M(t)||(x)dt 
 \\& \leq  ||\phi||_{\infty} 
 C_1 \e \int_{t_1}^{t_2}\int_{\R^n}  \big| H(x,M(t))\big|  d||M(t)||(x)dt + ||\phi||_{\infty}(C_1\e)^2 ||M(0)||(\R^n) (t_2-t_1).
  \end{split}
\end{equation}
For the first term, using \eqref{weakmcfequation} with $\phi$ being the constant function equal $1$ everywhere,
\begin{equation}
\begin{split}
 \int_{t_1}^{t_2}\int_{\R^n}  \big| H(x,M(t))\big|  d||M(t)||(x)dt
 & \leq \int_{t_1}^{t_2} \int_{\R^n}  \left( \big| H(x,M(t))\big|^2 +1 \right) d||M(t)||(x)dt
 \\& \leq (||M(t_1)||(\R^n)-||M(t_2)||(\R^n)) +  (t_2-t_1)||M(0)||(\R^n).
 \end{split}
\end{equation}
It yields,
\begin{equation}\label{hhepsilon4}
    \begin{split}
 & \Big| \int_{t_1}^{t_2}\int_{\R^n} \phi(x) |H(x,M(t))|^2 d||M(t)||(x)dt - \int_{t_1}^{t_2}\int_{\R^n} \phi(x) |H_{\e}(x,M(t))|^2 d||M(t)||(x)dt \Big|
 \\&\leq  ||\phi||_{\infty} C_1 \e 
  \Bigl( (||M(t_1)||(\R^n)-||M(t_2)||(\R^n)) + 
 (t_2-t_1)||M(0)||(\R^n) +
 C_1 ||M(0)||(\R^n) (t_2-t_1) \Bigr).
\end{split}
\end{equation}
Finally, we deduce from \eqref{weakmcfequation}, \eqref{hhepsilon3} and \eqref{hhepsilon4} that:
\begin{equation}\label{approximatebrakke1}
    \begin{split}
 & \Big| ||M(t_2)||(\phi)- ||M(t_1)||(\phi) + \int_{t_1}^{t_2}\int_{\R^n} \phi(x) |H_{\e}(x,M(t))|^2 - \nabla\phi(x)\cdot H_{\e}(x,M(t)) d||M(t)||(x)dt \Big| \\&\leq
||\phi||_{\infty}  C_1 \e (||M(t_1)||(\R^n)-||M(t_2)||(\R^n)) + ||\phi||_{C^1} C_1 \e ||M(0)||(\R^n) (2+C_1)(t_2-t_1)
\\& \leq |\phi||_{\infty}  C_1 \e (||M(t_1)||(\R^n)-||M(t_2)||(\R^n)) + ||\phi||_{C^1}c_3 \e(t_2-t_1) .
\end{split}
\end{equation}
where  $c_3 = C_1 ||M(0)||(\R^n) (2+C_1)$, it completes the proof.
\end{proof}
From \eqref{approximatebrakke1}, and the definition of the bounded Lipschitz distance we infer that 
\begin{equation}\label{approximatebrakke2}
    \begin{split}
 & \Big| ||V_h(t_2)||(\phi)- ||V_h(t_1)||(\phi) + \int_{t_1}^{t_2}\int_{\R^n} \phi(x) |H_{\e}(x,M(t))|^2 - \nabla\phi(x)\cdot H_{\e}(x,M(t)) d||M(t)||(x)dt \Big|
\\& \leq 2\lip(\phi) \max\limits_{t\in\lbrace t_1,t_2\rbrace}\Delta(M(t),V_h(t)) + |\phi||_{\infty}  C_1 \e (||M(t_1)||(\R^n)-||M(t_2)||(\R^n)) + ||\phi||_{C^1}c_3 \e(t_2-t_1) .
\end{split}
\end{equation}
We are now left with proving the following estimate: 
\begin{equation}\begin{split}
 & \Big|\int_{\R^n} \phi(x) |H_{\e}(x,M(t))|^2 - \nabla\phi(x)\cdot H_{\e}(x,M(t)) d||M(t)||(x)  \\& -
 \int_{\R^n} \phi(x) |H_{\e}(x,V_h(t))|^2 - \nabla\phi(x)\cdot H_{\e}(x,V_h(t)) d||V_h(t)||(x) 
 \Big| \leq ||\phi||_{\infty}C(\e+\frac{h}{\e^3})
 \end{split}
\end{equation}
for some constant $C$ to be specified.\\
To do so, let $\e\in(0,1)$, $\phi \in C_c^2(\R^n,\R^+)$. If we define  for every $V\in V_d(\R^n)$ and $x\in\R^n$:
\begin{equation}\label{phiepsilon0}
\varphi_{\e}(x,V):=-\phi(x) |H_{\e}(x,V)|^2+\nabla\phi(x)\cdot H_{\e}(x,V),                                                                                              \end{equation}
we have the following result:
\begin{prop}\label{mtov1}
Let $\Omega$ be a bounded open set of $\R^n$, $M$ a closed $C^2$ $d$-submanifold in $\Omega$ and $V_h$ a volumetric discretization of $M$ of parameter $h$ defined as in \ref{dfnvolumetric}.\\
If $h$ satisfies: $2h\leq \gamma \e$, for a constant $\gamma$ satisfying:
\begin{equation}
  \gamma \leq \max\limits_{x\in M} \lambda(x)^{-1}
  \quad \text{and} \,\, \gamma \leq (8(1+C_0^{d/2}))^{-1}
\end{equation}
where $\lambda(x)$ is the maximal principal curvature at $x$ and $C_0$ is the Ahlfors regularity constant of $M$. One has,
\begin{equation}
\Big| \int_{\R^n} \varphi_{\e}(\cdot,M) d||M|| -  \int_{\R^n} \varphi_{\e}(\cdot,M) d||V_h||\Big| \leq c_4 ||\phi||_{C^2} \frac{h}{\e^3},
\end{equation}
where $c_4=(2\gamma^{-1}(c_5^2+c_5)+2c_5c_6 + c_6) ||M(0)||(\R^n)$ and $c_5,c_6$ are constants  depending only on $\rho$, $\xi$, $C_0$ and $d$ and are defined in Lemma \ref{hebounds}.
\end{prop}
\begin{proof}
By Proposition \ref{volumapp}, we have: 
\begin{equation}
    \Big| \int_{\R^n} \varphi_{\e}(\cdot,M) d||M|| -  \int_{\R^n} \varphi_{\e}(\cdot,M)d||V_h||\Big| \leq  \underset{\R^n}{\lip}(\varphi_{\e}(M)) \Delta(||M||,||V_h||)  \leq h \underset{\R^n}{\lip}(\varphi_{\e}(M) )||M||(\R^n).
\end{equation}
Note that on the boundary of $\supp M^{\e}$, the quantity $||M||\ast\xi_{\e}$ might tend to $0$ faster that  $\delta M \ast \rho_{\e}$, this would make the norm of $H_{\e}$ explodes, so does $\lip(\varphi_{\e})$. As both $\supp M$ and $\supp V_h$ are included in $M^h$, we will introduce a cut function to measure the difference only on a small neighborhood containing $M^h$. We define a cut function  $\psi$ as follows:\\
We choose $h, \gamma$ satisfying the assumptions of the Proposition, we set:
$$
\psi(x) = \left\{
    \begin{array}{ll}
        1 & \mbox{if} \, x\in  (\supp M)^{h} \\
        0 & \mbox{if} \, x\in  \R^n\setminus \supp M^{\gamma\e}
        \end{array}
\right. \quad \text{and} \,\, \lip(\psi) \leq \frac{1}{\gamma\e-h}
$$
due to the choice of $\gamma$, $\psi$ is well defined, plugging the cutoff function into the expression: 
\begin{equation}\label{approximatebrakke4}
    \begin{split}
        &\Big| \int_{\R^n} \varphi_{\e}(\cdot,M) d||M|| -  \int_{\R^n} \varphi_{\e}(\cdot,M) d||V_h||\Big|
        \\&  = \Big| \int_{\R^n} \psi(\cdot)\varphi_{\e}(\cdot,M) d||M|| -  \int_{\R^n} \psi(\cdot)\varphi_{\e}(\cdot,M) d||V_h||\Big| 
        \\& \leq h \lip(\psi\varphi_{\e}(M)) ||M||(\R^n).
    \end{split}
\end{equation}
We have: $\lip(\psi) \leq \frac{1}{\gamma\e-h} \leq \frac{2}{\gamma}\e^{-1}$, this yields(omitting the dependence of $\varphi_{\e}$ and $H_{\e}$ on $M$ for a while) 
\begin{equation}\label{phiepsilon1}
\begin{split}
\lip(\psi\varphi_{\e}) \leq \max \psi \lip_{M^{\gamma\e}}\varphi_{\e} + \lip\psi \max_{M^{\gamma\e}} \varphi_{\e} \leq \lip_{M^{\gamma\e}} \varphi_{\e} + \frac{2}{\gamma}\e^{-1} \max_{M^{\gamma\e}} \varphi_{\e}  
\end{split}
\end{equation}
and,
\begin{equation}
\max_{M^{\gamma\e}} \varphi_{\e} \leq \max \phi (\max_{M^{\gamma\e}} |H_{\e}|^2) + \max \nabla \phi \max_{M^{\gamma\e}} |H_{\e}| 
\end{equation}
finally,
\begin{equation}
    \lip_{M^{\gamma\e}} \varphi_{\e} \leq \lip \phi (\max_{M^{\gamma\e}} H_{\e})^2 + \max \phi (\lip_{M^{\gamma\e}} |H_{\e}|^2) + \lip(\nabla\phi)\max_{M^{\gamma\e}} H_{\e} + \max(\nabla \phi) (\lip_{M^{\gamma\e}} H_{\e})
\end{equation}
by Lemma \ref{hebounds} that we will state right after, and for $V=M$, we can assert that there exist two constants $c_5$ and $c_6$ depending only on $\rho$, $\xi$, $C_0$ and $d$ such that:
\begin{equation}
 \max_{M^{\gamma\e}}|H_{\e}|\leq c_5 \e^{-1}, \quad \lip_{M^{\gamma\e}}(H_{\e})\leq c_6 \e^{-2}.
\end{equation}
Consequently,
\begin{equation}\label{phiepsilon2}
 \max\limits_{M^{\gamma\e}} \varphi_{\e} \leq ||\phi||_{\infty} c_5^2 \e^{-2} + ||\nabla \phi||_{\infty}c_5\e^{-1},
\end{equation}
using $\lip\limits_{M^{\gamma\e}}(|H_{\e}|^2)\leq 2 \max\limits_{M^{\gamma\e}}|H_{\e}| \lip\limits_{M^{\gamma\e}}(H_{\e})$ we obtain
\begin{equation}\label{phiepsilon3}
\begin{split}
 \lip_{M^{\gamma\e}} \varphi_{\e} &\leq ||\nabla\phi||_{\infty} c_5^2 \e^{-2}  + ||\phi||_{\infty}2c_5 \e^{-1} c_6 \e^{-2} + ||\nabla^2\phi||_{\infty} c_5\e^{-1} + ||\nabla\phi||_{\infty}c_6\e^{-2}.
 \end{split}
\end{equation}
From \eqref{phiepsilon1}, \eqref{phiepsilon2} and \eqref{phiepsilon3}, using $\gamma \leq 1$, we get
\begin{equation}
\begin{split}
 \lip(\psi\varphi_{\e}(M)) & \leq 2\gamma^{-1}\left( ||\phi||_{\infty} c_5^2 \e^{-2} + ||\nabla \phi||_{\infty}c_5\e^{-1} \right)\\&  + 
 ||\nabla\phi||_{\infty} c_5^2 \e^{-2}  + ||\phi||_{\infty}2c_5 \e^{-1} c_6 \e^{-2} + ||\nabla^2\phi||_{\infty} c_5\e^{-1} + ||\nabla\phi||_{\infty}c_6\e^{-2}
 \\&
 \leq ||\phi||_{C^2}\e^{-3}\left( 2 \gamma^{-1}\left(c_5^2 +c_5\right)+2c_5c_6+c_6\right).
\end{split}
\end{equation}
\eqref{approximatebrakke4} gives,
\begin{equation}
\begin{split}
        \Big| \int_{\R^n} \varphi_{\e}(\cdot,M) d||M|| -  \int_{\R^n} \varphi_{\e}(\cdot,M) d||V_h||\Big| \leq c_4  ||\phi||_{C^2} \frac{h}{\e^3}
\end{split}
\end{equation}
where $c_4= (2\gamma^{-1}(c_5^2+c_5)+2c_5c_6 + c_6) ||M||(\R^n)$.
\end{proof}
The proof of the following lemma is inspired by \cite[prop 4.6]{br}. 
\begin{lemma}\label{hebounds}
Let $V$ be a $d$-Ahlfors regular varifold in $\R^n$ with Ahlfors constant $C_0$. Let $\e,\gamma \in (0,1)$ with $\displaystyle \gamma\leq (8(1+C_0^{2/d}))^{-1}$, we have: 
\begin{equation}
 \max\limits_{V^{\gamma\e}} |H_{\e}(\cdot,V)| \leq c_5 \e^{-1},  \quad \lip\limits_{V^{\gamma\e}}(H_{\e}(\cdot,V))\leq c_6 \e^{-2},
\end{equation}
where $c_5,c_6$ are two constants depending on  $\rho,\xi, C_0$ and $d$.
\end{lemma}
\begin{proof} 
We start with $\max\limits_{V^{\gamma\e}} |H_{\e}|$, we first prove that there exists a constant $c_7$ such that:
$$\e^{n} \V \ast \xi_{\e}(z) \geq c_7 \e^{d} \quad \forall z \in V^{\gamma\e}.$$
Denote $$\beta = \min \big\lbrace \xi(s) \big| s\in \big[ \frac{C_0^{-2/d}}{4}, \frac{1}{2} \big]\big\rbrace > 0.$$
Let $z\in V^{\gamma\e}$, let $x\in \supp\V$ such that $|x-z|\leq \gamma\e$, we write
\begin{equation*}
\begin{split}
\e^{n} \V \ast \xi_{\e}(z) &\geq \beta \left( \V(B(z,\frac{\e}{2}))- \V(B(z,C_0^{-2/d}\frac{\e}{4})) \right)
\\&
\geq \beta \left( \V\left(B(x,\frac{\e}{2}-|x-z|)\right)- \V\left(B(x,C_0^{-2/d}\frac{\e}{4}+|x-z|)\right) \right)
\\&
\text{using the Ahlfors property }
\\&
\geq \beta \left( C_0^{-1} \left( \frac{\e}{2}-|x-z|\right)^d 
- C_0 \left( C_0^{-2/d}\frac{\e}{4} +|x-z|\right)^d
\right) \\& 
\geq \beta C_0^{-1}2^{-d} \left( (\e-2|x-z|)^d - (\frac{\e}{2} +2C_0^{2/d}|x-z|)^d\right) \\&
\geq \beta C_0^{-1}2^{-d} \left( (\e-2\gamma\e)^d - (\frac{\e}{2} +2C_0^{2/d}\gamma\e)^d\right) \\&
\geq \beta C_0^{-1}2^{-d}\e^{d} \left( (1-2\gamma)^d - (\frac{1}{2} +2C_0^{2/d}\gamma)^d\right) \\&
\text{ using "$a^d-b^d\geq (a-b)a^{d-1}$ when $a\geq b \geq 0$" we get: }
\\& \geq \beta C_0^{-1} 2^{-d}\e^{d} \left((\frac{1}{2}-2\gamma(1+C_0^{2/d}))(1-2\gamma)^{d-1}\right)
\\&  \geq \beta C_0^{-1} 2^{-d}\e^{d} \left(\frac{1}{2} \right)^{d+1}
\geq \beta C_0^{-1} 2^{-2d-1}\e^{d}, 
\end{split}
\end{equation*}
thus,
\begin{equation}\label{c_7lower}
\e^{n} \V \ast \xi_{\e}(z) \geq  c_7 \e^{d},
\end{equation}
with $c_7=\beta C_0^{-1} 2^{-2d-1}$.\\
Let $x \in \supp \V$ such that $|x-z|\leq \gamma\e \leq \e$, we have 
\begin{equation*}
\begin{split}
    \big|\e^{n} (\delta V \ast \rho_{\e})(z) \big| &=  \big| \int S(\nabla \rho_{\e}(z-y)) dV(y,S)  \big|
    \leq \e^{-1}||\rho'||_{\infty} \V\left(B(z,\e)\right) \\&
    \leq \e^{-1}||\rho'||_{\infty} \V\left(B(x,2\e)\right)
 \leq \e^{-1}||\rho'||_{\infty} C_0 (2\e)^{d} = 2^d ||\rho'||_{\infty} C_0 \e^{d-1},
\end{split}
\end{equation*}
thus, for any $z\in \R^n$,
\begin{equation*}
 \big|\e^{n} (\delta V \ast \rho_{\e})(z) \big| \leq 2^d ||\rho'||_{\infty} C_0 \e^{d-1}.
\end{equation*}
Consequently,
\begin{equation}
 \max\limits_{z\in V^{\gamma\e}} |H_{\e}(z,V)|=\max\limits_{z\in V^{\gamma\e}} \frac{|(\delta V \ast \rho_{\e})(z)|}{|(\V \ast \xi_{\e})(z)|}  \leq \frac{2^d ||\rho'||_{\infty} C_0 \e^{d-1}}{\beta C_0^{-1} 2^{-2d-1} \e^{d}} \leq \beta^{-1}C_0^2 2^{3d+1}||\rho'||_{\infty} \e^{-1},
\end{equation}
for $c_5=\beta^{-1}C_0^2 2^{3d+1}||\rho'||_{\infty}$, we conclude the proof of the first part of our Lemma.\\\\
We now deal with $\lip\limits_{V^{\gamma\e}}(H_{\e})$, for any $x,z \in  V^{\gamma\e}$ we have
\begin{equation*}
\begin{split}
    \e^{n} \big| \V\ast\xi_{\e}(z) - \V\ast\xi_{\e}(x)  \big| 
    & = 
    \Big| \int \xi\left(\frac{z-y}{\e}\right) d\V(y) -\int \xi \left(\frac{x-y}{\e}\right) d\V(y) \Big| 
    \\&\leq
    |x-z|\e^{-1} ||\xi'||_{\infty}  \V\left( B(x,\e) \cup B(z,\e) \right) 
    \\&\leq 
    |x-z| \e^{-1} ||\xi'||_{\infty} \V\left( B(x',2\e) \cup B(z',2\e) \right) 
    \text{for some $x',z' \in supp \V$}
    \\&\leq 
    |x-z| \e^{-1} ||\xi'||_{\infty} \left(2 C_0 (2\e)^{d} \right) \leq 2^{d+1} C_0 ||\xi'||_{\infty} |x-z| \e^{d-1}.
    \end{split}
\end{equation*}
Same computations give 
\begin{equation}
     \e^{n} \big| \delta V \ast\rho_{\e}(z) - \delta V \ast\rho_{\e}(x)  \big| \leq 2^{d+1} C_0 ||\rho''||_{\infty} |x-z| \e^{d-2},
\end{equation}
therefore, 
\begin{equation*}
\begin{split}
    \Big| H_{\e}(z,V) - H_{\e}(x,V) \Big| & 
    = \Big| \frac{\delta V \ast\rho_{\e}(z)}{\V\ast\xi_{\e}(z)} - \frac{\delta V \ast\rho_{\e}(x)}{\V\ast\xi_{\e}(x)} \Big|
    \\& \leq
    \frac{\big|\delta V \ast\rho_{\e}(z)-\delta V \ast\rho_{\e}(x)  \big|}{\V \ast\xi_{\e}(z)} + |\delta V \ast\rho_{\e}(x)| \Big| \frac{1}{\V\ast\xi_{\e}(z)} - \frac{1}{\V\ast\xi_{\e}(x)}\Big| \\&\leq
    \frac{2^{d+1} C_0 ||\rho''||_{\infty} |x-z| \e^{d-2} }{\e^n\V \ast\xi_{\e}(z)} +  2^d ||\rho'||_{\infty} C_0  \e^{d-1} \frac{2^{d+1} C_0 ||\xi'||_{\infty}  |x-z|\e^{d-1}}{\e^{2n}\V \ast\xi_{\e}(z) \V \ast\xi_{\e}(x)} 
    \\&\leq 
    \frac{2^{d+1} C_0 ||\rho''||_{\infty} |x-z| \e^{d-2} }{\beta C_0^{-1} 2^{-2d-1}\e^d} +  2^d ||\rho'||_{\infty} C_0  \e^{d-1} \frac{2^{d+1} C_0 ||\xi'||_{\infty}  |x-z|\e^{d-1}}{\beta C_0^{-1} 2^{-2d-1} \e^d \beta C_0^{-1} 2^{-2d-1} \e^d} 
    \\& \leq \beta^{-1} C_0^2 2^{3d+1}||\rho'||_{\infty} |x-z| \e^{-2}  + 2^{6d+3} C_0^4 \beta^{-2} ||\rho'||_{\infty}||\xi'||_{\infty} |x-z|\e^{-2}. 
    \end{split}
\end{equation*}
In conclusion,
\begin{equation}
 \lip\limits_{V^{\gamma\e}} (H_{\e}(\cdot,V))\leq c_6  \e^{-2},
\end{equation}
for $c_6= \beta^{-1} C_0^2 2^{3d+1}||\rho'||_{\infty} \left(1+ \beta^{-1} C_0^2 2^{3d+2} ||\xi'||_{\infty} \right) $, and this finishes the proof.
\end{proof}
\begin{prop}\label{mtov2}
Under the same assumptions of Proposition \ref{mtov1}, if we assume in addition: 
\begin{equation}
   \displaystyle \beta > \gamma 2^{3d}C_0^2(\lip(\xi)+1).
\end{equation}
Then,
\begin{equation}
 \Big| \int_{\R^n} \varphi_{\e}(\cdot,M) d||V_h|| -  \int_{\R^n} \varphi_{\e}(\cdot,V_h) d||V_h||\Big| \leq c_8||\phi||_{C^1} \frac{h}{\e^3},
\end{equation}
where $\varphi_{\e}$ was previously defined in \eqref{phiepsilon0} and $c_8$ is defined in \eqref{c_8lower}.
\end{prop}
\begin{proof}
We start by estimating: $\big| H_{\e}(z,M) - H_{\e}(z,V_h)\big|$ on $ \supp M^{h}.$ To do so, let $z\in \supp M^h$, let $x \in \supp M$ such that 
$|x-z|\leq h$, we have by \cite[Lemma 4.4]{blm1} with $B= B(x,\e+|x-z|)$ and using that $\xi_{\e}$ is $\e^{-n-1} \lip(\xi)$-Lipschitz, 
\begin{equation}
 \begin{split}
  \e^n||V_h||\ast\xi_{\e}(z) +\e^{-1}\lip(\xi)\left( \Delta_{B}(\M,||V_h||) +|x-z|||M||(B) \right) \geq \e^n \M \ast \xi_{\e}(x)
 \end{split}
\end{equation}
this yields, by Proposition \ref{volumapp} and \eqref{c_7lower} (with $V=M$) we have
\begin{equation}
 \begin{split}
  \e^n||V_h||\ast\xi_{\e}(z) & \geq -2\e^{-1}\lip(\xi) h \M(B)  + \beta C_0^{-1}2^{-2d-1} \e^{d} 
  \\& > -2\e^{-1}\lip(\xi) h C_0(2\e)^d +  \beta C_0^{-1}2^{-2d-1} \e^{d}
  \\& \geq \e^{d}(-\gamma \lip(\xi) 2^d + \beta C_0^{-1}2^{-2d-1}).
 \end{split}
\end{equation}
Thus, if we choose $\gamma$ satisfying $\displaystyle \beta >  \gamma C_0^2 2^{3d+1} \lip(\xi)$, we obtain:
\begin{equation}\label{c_9lower}
  \e^n||V_h||\ast\xi_{\e}(z) \geq c_9 \e^{d}
\end{equation}
where $c_9= \beta - \gamma C_0^2 2^{3d+1} \lip(\xi)$. Next, using Proposition \ref{volumapp}, \eqref{c_7lower}(with $V=M$) and \eqref{c_9lower}, and denoting by $C_2$ the Lipschitz constant of the map $x\in M \mapsto T_xM$,
\begin{equation}
\begin{split}
 \big| H_{\e}(z,M) - H_{\e}(z,V_h)\big|
 & = \Big| \frac{\delta M \ast\rho_{\e}(z)}{\M\ast\xi_{\e}(z)} - \frac{\delta V_h \ast\rho_{\e}(z)}{||V_h||\ast\xi_{\e}(z)} \Big|
\\& \leq 
|\delta M \ast\rho_{\e}(z)| \Big| \frac{1}{\M\ast\xi_{\e}(z)} - \frac{1}{||V_h||\ast\xi_{\e}(z)}\Big| +
\frac{\big|\delta M \ast\rho_{\e}(z)-\delta V_h \ast\rho_{\e}(z)  \big|}{||V_h|| \ast\xi_{\e}(z)} 
\\&\leq
\e^{-1}||\rho'||_{\infty}\M(B) \frac{ h ||\xi'||_{\infty} \M(B)}{ \e^{2n}\M\ast\xi_{\e}(z) ||V_h||\ast\xi_{\e}(z)} +
\frac{\e^{-2}||\rho''||_{\infty}h(1+2C_2)\M(B)}{\e^n ||V_h|| \ast\xi_{\e}(z)}
\\& \leq
 \frac{ \e^{-1}||\rho'||_{\infty}||\xi'||_{\infty} C_0 2^d \e^d h C_0 2^d \e^d}{c_7 \e^d c_9 \e^d} +
\frac{\e^{-2}||\rho''||_{\infty}h(1+2C_2)C_0 2^d\e^d}{c_9 \e^d}
\\& \leq 
\frac{h}{\e}\frac{||\rho'||_{\infty}||\xi'||_{\infty}2^{2d}C_0^2}{c_7c_9}+ \frac{h}{\e^2}\frac{2^d||\rho''||_{\infty}(1+2C_2)C_0}{c_9}.
\end{split}
\end{equation}
Thus, for  
\begin{equation}\label{c_10lower}
 c_{10} = 
\frac{||\rho'||_{\infty}||\xi'||_{\infty}2^{2d}C_0^2}{c_7c_9}+ \frac{2^d||\rho''||_{\infty}(1+2C_2)C_0}{c_9},
\end{equation}
we have  
\begin{equation}\label{stabhepsilon}
 \big| H_{\e}(z,M) - H_{\e}(z,V_h)\big| \leq c_{10}\frac{h}{\e^2} \quad \forall z  \in \supp(M)^h.
\end{equation}
Finally, using \eqref{stabhepsilon}, Lemma \ref{hebounds} with $V=M$ and $||V_h||(\R^n) = ||M||(\R^n)$ we obtain
\begin{equation*}
\begin{split}
&\Big| \int_{\R^n} \varphi_{\e}(z,M) d||V_h||(z) -  \int_{\R^n} \varphi_{\e}(z,V_h) d||V_h||(z) \Big| 
\\& \leq \int_{\R^n} \phi(z) \Big| |H_{\e}(z,M)|^2-|H_{\e}(z,V_h)|^2 \Big|d||V_h||(z) + \int_{\R^n} |\nabla\phi(z)| \big| H_{\e}(z,M)-H_{\e}(z,V_h) \big| d||V_h||(z)
\\& \leq  c_{10} \frac{h}{\e^2}  \int_{\R^n}\phi(z)\left(|H_{\e}(z,M)|+|H_{\e}(z,V_h)|\right)d||V_h||(z)+ c_{10} \frac{h}{\e^2} \int_{\R^n} |\nabla\phi(z)|d||V_h||(z) 
\\&\leq  2c_{10} \frac{h}{\e^2}||\phi||_{\infty}  \int_{\R^n} |H_{\e}(z,M)| d||V_h||(z)  + c_{10}||\phi||_{\infty} \frac{h^2}{\e^4} ||M||(\R^n) +  c_{10} \frac{h}{\e^2} ||\nabla\phi||_{\infty}||M||(\R^n)
\\&
\leq ||\phi||_{\infty} 2 c_5c_{10} ||M||(\R^n) \frac{h}{\e^3} + ||\phi||_{\infty}||M||(\R^n)c_{10} \frac{h^2}{\e^4} + c_{10}  ||\nabla\phi_{\infty}||M||(\R^n) \frac{h}{\e^2}.
\end{split}
\end{equation*}
Thus, 
 \begin{equation}
  \Big| \int_{\R^n} \varphi_{\e}(\cdot,M) d||V_h|| -  \int_{\R^n} \varphi_{\e}(\cdot,V_h) d||V_h||\Big| 
 \leq c_8 ||\phi||_{C^1}\frac{h}{\e^3},
 \end{equation}
where
\begin{equation}\label{c_8lower}
c_8= c_{10}||M||(\R^n)||\phi||_{C^1}\left( 2c_5 + 1  \right),       
\end{equation}
this concludes the proof.
\end{proof}

\begin{proof}[Proof of the Theorem] Let $\Omega$ be a convex bounded open set of $\R^n$, $\e\in(0,1)$, $M$ a closed $C^3$ $d$-submanifold in $\Omega$ and $M(t)$ its mean curvature flow, defined for $t\in[0,T]$. Let  $V_h(t)$ be a volumetric discretization of $M(t)$ of parameter $h$ for all $t\in[0,T]$. Assume the assumptions of Theorem \ref{thmappbrakke} on $\gamma$ and $h$ fulfilled, and that $C_0$, $C_1$ and $C_2$ are uniform on $t$ (this is possible as both the constants and the flow evolves continuously in time w.r.t the $C^3$ distance on the space of $d$-submanifolds and the interval $[0,T]$ is compact). Using $||M(t)||(\R^n) \leq ||M(0)||(\R^n)$ and Propositions \ref{mtov1}, \ref{mtov2}, both with $M=M(t)$, we have:
\begin{equation}
\begin{split}
 &\Big| \int_{t_1}^{t_2}\int_{\R^n} -\phi|H_{\e}(z,M(t))|^2 + \nabla\phi(z) \cdot H_{\e}(z,M(t))  d||M(t)||dt \\& -  \int_{t_1}^{t_2}\int_{\R^n}-\phi|H_{\e}(z,V_h(t))|^2 + \nabla\phi(z) \cdot H_{\e}(z,V_h(t))  d||V_h(t)||dt \Big| 
 \\& \leq (c_4+c_8)(t_2-t_1)||\phi||_{C^2} \frac{h}{\e^3},
\end{split}
\end{equation}
for all $0\leq t_1 \leq t_2 \leq T$ and $\phi \in C^2_c(\R^n,\R^+)$. Combining with \eqref{approximatebrakke2} we get:
\begin{equation}\label{approximatebrakke3}
    \begin{split}
 & \Big| ||V_h(t_2)||(\phi)- ||V_h(t_1)||(\phi) + \int_{t_1}^{t_2}\int_{\R^n} \phi(x) |H_{\e}(x,V_h(t))|^2 - \nabla\phi(x)\cdot H_{\e}(x,V_h(t)) d||V_h(t)||(x)dt \Big| \\&
 \leq 2\lip(\phi)\max\limits_{t\in\lbrace t_1,t_2\rbrace}  \Delta(M(t),V_h(t)) +
||\phi||_{\infty}  C_1 \e (||M(t_1)||(\R^n)-||M(t_2)||(\R^n)) + C(t_2-t_1)||\phi||_{C^2}(\e+ \frac{h}{\e^3}),
\end{split}
\end{equation}
for all $0\leq t_1 \leq t_2 \leq T$ and $\phi \in C^2_c(\R^n,\R^+)$, where $C= c_3+c_4+c_8 $, and we conclude the proof the of our Theorem.
\end{proof}

\section{List of constants used in the paper}
\begin{itemize}
 \item $C_0$ The Ahlfors regularity constant.
 \item $C_1$ Measures how well $H_{\e}$ approximates $H$ \ref{hhepsilon1}.
 \item $C_2$ The Lipschitz constant of the map $y \mapsto T_yM$.
 \item $\gamma$ any constant satisfying: 
 \begin{equation*}\begin{split}
 &\gamma \leq \min \Big\lbrace (8(1+C_0^{2/d}))^{-1} , \max\limits_{x\in M(t)}(\lambda(x))^{-1} , \beta \left(2^{3d}C_0^2(\lip(\xi)+1)\right)^{-1} \Big\rbrace, 
 \end{split}\end{equation*}
 $\lambda(x)$ \text{is the biggest principal curvature at $x$}
 \item $ \displaystyle \beta =  \min \big\lbrace \xi(s) \big| s\in \big[ \frac{C_0^{-2/d}}{4}, \frac{1}{2} \big] \big\rbrace > 0 $.
 \item $\displaystyle c_3 = C_1 ||M(0)||(\R^n) (2+C_1) $. \eqref{approximatebrakke2}
 \item $\displaystyle c_4 =(2\gamma^{-1}(c_5^2+c_5)+2c_5c_6 + c_6) ||M||(\R^n) $
 \item $\displaystyle c_5 =\beta^{-1}C_0^2 2^{3d+1}||\rho'||_{\infty}$
 \item $\displaystyle c_6 =\beta^{-1} C_0^2 2^{3d+1}||\rho'||_{\infty} \left(1+ \beta^{-1} C_0^2 2^{3d+2} ||\xi'||_{\infty} \right) $
 \item $\displaystyle c_7 = \beta C_0^{-1} 2^{-2d-1} $ \eqref{c_7lower}.
 \item $\displaystyle c_8= c_{10}||M||(\R^n)||\phi||_{C^1}\left( 2c_5 + 1  \right) $ \eqref{c_8lower}.
 \item $\displaystyle c_9=\beta - \gamma C_0^2 2^{3d+1} \lip(\xi)$ \eqref{c_9lower}.
 \item $\displaystyle c_{10} = 
\frac{||\rho'||_{\infty}||\xi'||_{\infty}2^{2d}C_0^2}{c_7c_9}+ \frac{2^d||\rho''||_{\infty}(1+2C_2)C_0}{c_9} $ \eqref{c_10lower}.
 \item $\displaystyle C = c_3+c_4+c_8 $ appears in the Brakke approximate inequality \ref{thmappbrakke}.
 \item $\displaystyle C'=||M(0)||(\R^n)(2+C_1)+CT $ appears in the weak version of the Brakke approximate inequality \ref{thmappbrakke1}.
\end{itemize}

\newpage
\bibliographystyle{abbrv}
\bibliography{these}
\end{document}